\documentclass [12pt,a4paper,reqno]{amsart}
\input amssymb.sty

\def\XXint#1#2#3{{\setbox0=\hbox{$#1{#2#3}{\int}$}
     \vcenter{\hbox{$#2#3$}}\kern-.5\wd0}}

\usepackage{amsmath,amsthm}
\usepackage{amsfonts}
\usepackage{datetime}
\allowdisplaybreaks

\newcounter{lemma}[section]

\newcounter{corol}[section]

\newcounter{rem}[section]

\newcounter{theo}[section]

\newcounter{propo}[section]

\begin{document}

\title[ACP of spatial open discrete mappings]{
Absolute continuity on paths of spatial open discrete mappings}

\author[A. Golberg and E. Sevost'yanov]{Anatoly Golberg and Evgeny Sevost'yanov}

\begin{abstract}
We prove that open discrete mappings of Sobolev classes $W_{\rm loc}^{1, p},$ $p>n-1,$ with locally integrable inner dilatations admit
$ACP_p^{\,-1}$-property, which means that these mappings are absolutely continuous on almost all preimage paths with respect to $p$-module. In particular, our results
extend the well-known Poletski\u\i\ lemma for quasiregular
mappings. We also establish the upper bounds for $p$-module of such mappings in terms of integrals depending on the inner dilatations and arbitrary admissible functions.
\end{abstract}

\date{\today \hskip 3mm \currenttime \hskip 4mm (\texttt {ACPMCPM.tex})}

\maketitle

\bigskip

\bigskip
{\small {\textbf {2010 Mathematics Subject Classification: }
Primary: 30C65, 26B05, Secondary: 31B15}}

\bigskip
{\small {\textbf {Key words:} $p$-module, absolute continuity, Sobolev classes, quasiconformal and quasiregular mappings}}

\bigskip

\medskip

\section{Introduction}

\subsection{} The absolute continuity on almost all lines of mappings is one of crucial properties of quasiconformal mappings and more general quasiregular mappings (see, e.g. \cite{MRV69}, \cite{Resh89}, \cite{Ric93}, \cite{Vai71}, \cite{Vuo88}). Later many authors have extended this important property to more general classes of continuous mappings of finitely dimensional domains including mappings of finite distortion, mappings with finite length distortion, mappings quasiconformal in the mean, mappings with controlled $p$-module and others (\cite{AC05}, \cite{BGMV03}, \cite{Gol05}, \cite{GS12}, \cite{HK14}, \cite{IR05}, \cite{IM01}, \cite{MRSY09}, \cite{Mikl08}, \cite{SS14}). It remains still open whether the generic mappings from Sobolev classes possess absolute continuity and related differential properties (cf. \cite{BKR07}, \cite{Gol09}, \cite{KR05}, \cite{Pol70}).

In this paper, we discuss absolute continuity of Sobolev mappings on paths. Due to terminology of \cite{MRSY09}, this class $ACP\subset ACL,$ where $ACL$ denotes the class of mappings absolutely continuous on lines. The $ACP$-property is a basic tool for studying geometric features of mappings (cf. \cite{AC05}, \cite{Fu57}). It was shown in \cite{Pol70}, that quasiregular mappings provide absolute continuity for preimages; namely, if $f:D\rightarrow {\mathbb R}^n,$ $n\ge 2,$ is a quasiregular mapping of a domain $D\subset\mathbb R^n,$ then for almost all curves
$\gamma_*\in f(D)$ a curve $\gamma,$ such that
$f\circ\gamma=\gamma_*,$ is absolutely continuous. This property can be treated as $ACP^{-1}$-property. For this property and the definition of mappings with finite length distortion we refer to \cite{MRSY04} and \cite[Ch.~8]{MRSY09}.

An extension of Poletski\u\i's result to general open discrete mappings of Sobolev class $W_{\rm loc}^{1,n}(D)$ is given \cite{Sev09}, under assumption that the $n$-dimensional Lebesgue measure of the branching set $m (B_f)=0$ and local integrability with appropriate degree of one of quasiconformality coefficients of mappings. Other resent extensions related to absolute continuity can be found in \cite{Cris08}, \cite{Cris10}, \cite{Guo}, \cite{Sal08}, \cite{SS14} and \cite{Ten14}.

\subsection{} Throughout this paper, $D$ will be a domain in
${\mathbb R}^n,$ $n\ge 2;$ $m$ denotes the Lebesgue measure in ${\mathbb R}^n.$
A mapping $f:D\rightarrow {\mathbb R}^n$ is {\it discrete} if
$f^{-1}(y)$ consists of isolated points for each $y\in{\mathbb R}^n,$
and $f$ is {\it open} if it maps open sets onto open sets. The
notation $f:D\rightarrow {\mathbb R}^n$ assumes that $f$ is continuous,
and all mappings are orientation
preserving, i.e., the topological index $i(y, f,G)>0$ for an
arbitrary domain $G\Subset D$ and
$y\in f(G)\setminus f(\partial G);$ see, e.g. \cite[II.2]{Resh89}. Let
$f:D\rightarrow {\mathbb R}^n$ be a mapping and suppose that there is a
domain $G\subset D,$ $\overline{G}\subset D,$ for which $
f^{\,-1}\left(f(x)\right)=\left\{x\right\}.$ Then the quantity
$\mu(f(x), f, G),$ which is referred to be the local topological
index, which does not depend on the choice of $G$ and is
denoted by $i(x, f).$

\medskip
A curve $\gamma$ in ${\mathbb R}^n$ $(\,\overline{{\mathbb R}^n}\,)$ is a continuous mapping $\gamma
:I\rightarrow{\mathbb R}^n\;(\,\overline{{\mathbb R}^n}\,),$ where $I$ is an interval in
${\mathbb R} .$ Its locus $\gamma(I)$ is denoted by $|\gamma|.$
Let $\Gamma$ be a family of curves $\gamma$ in ${\mathbb R}^n .$ A Borel
function $\rho:{\mathbb R}^n \rightarrow [0,\infty]$ is called {\it
admissible} for $\Gamma $ (abbr. $\rho \in {\rm adm}\, \Gamma $) if
$\int\limits_{\gamma} \rho(x)|dx| \ge 1$
for each locally rectifiable $\gamma\in\Gamma.$ For $p\ge 1,$
we define the quantity
\begin{equation*}
\mathcal M_p(\Gamma)\,=\,\inf\limits_{ \rho \in {\rm adm}\, \Gamma}
\int\limits_{{\mathbb R}^n} \rho^p(x) dm(x),
\end{equation*}
which is called {\it $p$-mo\-du\-le} of $\Gamma$ \cite[6.1]{Vai71} (see also \cite{Ric93}, \cite{Vuo88}).

\medskip
For the points of differentiability of $f$ we define
\begin{equation*}
l\left(f^{\,\prime}(x)\right)\,=\,\min\limits_{h\in {\mathbb
R}^n \backslash \{0\}} \frac {|f^{\,\prime}(x)h|}{|h|}\,,
||f^{\,\prime}(x)||\,=\,\max\limits_{h\in {\mathbb
R}^n \backslash \{0\}} \frac {|f^{\,\prime}(x)h|}{|h|}\,,
J(x,f)=\det f^{\,\prime}(x)\,,
\end{equation*}
and for any $x\in D$ and fixed $p\ge 1$ the $p$-inner dilatation  of $f$ by
\begin{equation*}
K_{I, p}(x,f)\quad =\quad\left\{
\begin{array}{rr}
\frac{|J(x,f)|}{{l\left(f^{\,\prime}(x)\right)}^p}, & J(x,f)\ne 0,\\
1,  &  f^{\,\prime}(x)=0, \\
\infty, & {\rm otherwise}.
\end{array}
\right.
\end{equation*}
The notation $K_I(x, f)=K_{I, n}(x, f)$ relates to the classical inner dilatation coefficient.

\medskip
Recall that a mapping $f:D\rightarrow {\mathbb R}^n$ is said to possess Lusin's
{\it $N$-pro\-per\-ty} if
$m\left(f\left(S\right)\right)=0$ whenever $m(S)=0$ for
$S\subset{\mathbb R}^n.$ Similarly, $f$ has {\it
$N^{-1}$--pro\-per\-ty} if $m\left(f^{\,-1}(S)\right)=0$ whenever
$m(S)=0.$

\subsection{}
Let $\Gamma$ be a curve family in $D,$ $\Gamma^{\,\prime}$ be
a curve family in ${\mathbb R}^n$ and  $q$ be a positive integer such
that the following is true. Suppose that for every curve
$\beta:I\rightarrow D$ in $\Gamma^{\,\prime}$ there exist curves
$\alpha_1,\ldots,\alpha_q$ in $\Gamma$ such that $f\circ
\alpha_j\subset \beta$ for all $j=1,\ldots,q,$ and for every $x\in
D$ and all $t\in I$ the equality $\alpha_j(t)=x$ holds at most
$i(x,f)$ indices $j.$

\medskip
\begin{theo}{}\label{th1A}
{\em Let $f:D\rightarrow {\mathbb R}^n$ be an open discrete mapping of the class
$W_{\rm loc}^{1, p},$ $n-1<p\le n,$ for which $N$ and $N^{\,-1}$-properties
hold. Suppose that $K_I(x, f)\in L_{\rm loc}^1.$ Then
\begin{equation}\label{eq5*AA}
\mathcal M_p(\Gamma^{\,\prime})\le \frac{1}{q}\int\limits_D K_{I,
p}(x,f)\cdot\rho^p(x)\, dm(x)
\end{equation}
for every $\rho\in {\rm adm\,}\Gamma.$ In particular,
\begin{equation*}
\mathcal M_p(f(\Gamma))\le \int\limits_D K_{I,
p}(x,f)\cdot\rho^p(x)\, dm(x)
\end{equation*}
for every family of curves $\Gamma$ in $D$ and $\rho\in {\rm
adm\,}\Gamma.$}
\end{theo}

\medskip
\begin{theo}{}\label{th1B}
{\em Let $f:D\rightarrow {\mathbb R}^n$ be an open discrete mapping of the class
$W_{\rm loc}^{1, p},$ $p>n,$ for which $N^{\,-1}$-property
holds. Assume that $K_I(x, f)\in L_{\rm loc}^1$ and $K_{I, p}(x,
f)\in L_{\rm loc}^1.$ Then the inequality (\ref{eq5*AA}) holds. }
\end{theo}

\section{Auxiliary results}

Let us first recall some necessary definitions, notations and statements.

\subsection{} We say that a property $\mathcal P$ holds for {\it $p$-almost every
($p$-a.e.)} curves $\gamma$ if $\mathcal P$ holds
for all curves except a family of zero $p$-module.

If $\gamma :I\rightarrow{\mathbb R}^n$ is a locally rectifiable
curve, then there is a unique nondecreasing length function
$l_{\gamma}$ of $I$ onto a length interval $I_{\gamma}\subset{\mathbb R}$ with a prescribed normalization $l
_{\gamma}(t_0)=0\in I_{\gamma},$ $t_0\in\Delta,$ such that $l
_{\gamma}(t)$ is equal to the length of the arc $\gamma
|_{[t_0,t]}$ of $\gamma$ for $t>t_0,$ $t\in I ,$ and $l
_{\gamma}(t)$ is equal to minus length of $\gamma |_{[t,t_0]}$ when
$t<t_0,$ $t\in I .$ Let $g: |\gamma |\rightarrow{\mathbb R}^n$ be a
continuous mapping and $\widetilde{\gamma}
=g\circ\gamma$ be locally rectifiable. Then there is a unique
nondecreasing function $L_{\gamma ,g}: I_{\gamma}\rightarrow I_{\widetilde{\gamma}}$ such that
$L_{\gamma ,g}\left(l_{\gamma}(t)\right) =
l_{\widetilde{\gamma}}(t)$ for all $t\in I.$ The curve $\gamma$ in
$D$ is called a (whole) {\it lifting} of the curve
$\widetilde{\gamma}$ in ${\mathbb R}^n$ under $f:D\rightarrow {\mathbb
R}^n$ and $\widetilde{\gamma} = f\circ\gamma.$

\medskip
Following \cite{SS14}, we say that a mapping $f:D\rightarrow{\mathbb R}^n$
satisfies the {\it $L^{(2)}_p$-pro\-pe\-r\-ty,} if, for
$p$-a.e. curve $\widetilde{\gamma}$ in $f(D),$ each lifting $\gamma$
of $\widetilde{\gamma}$ is locally rectifiable and the function
$L_{\gamma ,f}$ satisfies $N^{-1}$-pro\-per\-ty.

For any
closed rectifiable path $\gamma:[a,b]\rightarrow {\mathbb R}^n,$ we define
the length function $l_{\gamma}(t)$ by
$l_{\gamma}(t)=S\left(\gamma, [a,t]\right),$ where $S(\gamma,
[a,t])$ is the length of the path $\gamma|_{[a, t]}.$ Let
$\alpha:[a,b]\rightarrow {\mathbb R}^n$ be a rectifiable curve in
${\mathbb R}^n,$ $n\ge 2,$ and $l(\alpha)$ be its length. A {\it normal
representation} $\alpha^0$ of $\alpha$ is defined as a curve
$\alpha^0:[0, l(\alpha)]\rightarrow {\mathbb R}^n$ which is obtained
from $\alpha$ by changing parameter so that
$\alpha(t)=\alpha^0\left(S\left(\alpha, [a, t]\right)\right)$ for
every $t\in [0, l(\alpha)].$

\medskip
Suppose that $\alpha$ and $\beta$ are curves in ${\mathbb R}^n.$ Then
the notation $\alpha\subset\beta$ means that $\alpha$ is a subpath
of $\beta.$ In what follows, $I$ denotes either an open or  closed
or semi-open interval on the real axes. The following definition
is due to \cite[Ch. II]{Ric93}.

\medskip
Let $f:D\rightarrow {\mathbb R}^n$ be a mapping whose inverse $f^{-1}(y)$
does not contain a nondegenerate curve, $\beta:I_0\rightarrow
{\mathbb R}^n$ be a closed rectifiable curve and $\alpha:I\rightarrow
D$ be such that $f\circ \alpha\subset \beta.$ If the length function
$l_{\beta}:I_0\rightarrow [0, l(\beta)]$ is constant on $J\subset
I,$ then $\beta$ is also constant on $J$ and, consequently, the curve
$\alpha$ is constant on $J$ too. Thus, there is a unique function
$\alpha^{\,*}:l_\beta(I)\rightarrow D$ such that
$\alpha=\alpha^{\,*}\circ (l_\beta|_I).$ We say that $\alpha^{\,*}$
is the {\it $f$--re\-pre\-se\-n\-ta\-tion of $\alpha$ with respect
to $\beta$ } if $\beta=f\circ\alpha.$

\medskip
\begin{rem}\label{rem1}
Note that properties of $L_{\gamma, f}$ connected with the
length functions $l_{\gamma}(t)$ and $l_{\widetilde{\gamma}}(t),$
$\widetilde{\gamma}=f\circ\gamma,$ do not essentially depend on the
choice of $t_0\in [a, b).$ Moreover, one can take $t_0=a,$ because for any $t_0\in (a, b),$ we have $S(\gamma, [a,
t])=S(\gamma, [a, t_0])+l_{\gamma}(t).$ Further we put
$t_0=a$ and use the notion $l_{\gamma}(t)$ for the length of the
path $\gamma|_{[a, t]}$ whenever the curve $\gamma$ is closed.
\end{rem}

\medskip
Given a set $E$ in ${\mathbb R}^n$ and a curve $\gamma
:I\rightarrow {\mathbb R}^n,$ we identify $\gamma\cap E$ with
$\gamma\left(I\right)\cap E.$
If $\gamma$ is locally rectifiable, we set
\begin{equation*}
l\left(\gamma\cap E\right) =  m_1(E_ {\gamma}),
\end{equation*}
where
$E_ {\gamma} = l_{\gamma}\left(\gamma ^{-1}\left(E\right)\right).$
Note that
$E_ {\gamma} = \gamma _0^{-1}\left(E\right),$
where $\gamma _0 :I_{\gamma}\rightarrow {\mathbb R}^n$ is the
natural parametrization of $\gamma $ and
\begin{equation*}
l\left(\gamma\cap E\right) = \int\limits_I \chi
_E\left(\gamma\left(t\right) \right) |dx|:= \int\limits_{I_{\gamma}} \chi_{E_\gamma }(s)\, ds\,.
\end{equation*}

\subsection{} We shall use the following result established in \cite{SS14}.

\medskip
\begin{propo}\label{pr3}
{\em A mapping $f:D\rightarrow {\mathbb R}^n$ has
the $L^{(2)}_p$--pro\-per\-ty if and only if $f^{-1}(y)$ does not
contain a nondegenerate curve for every $y\in {\mathbb R}^n,$ and the
$f$-representation $\gamma^{\,*}$ is rectifiable and absolutely
continuous for $p$-a.e. closed curve
$\widetilde{\gamma}=f\circ\gamma.$}
\end{propo}

\medskip
The following three statements can be found in \cite{Pol70}; see Lemmas 1--3.

\medskip
\begin{propo}\label{pr1*0}
{\em Let $\gamma_1:I=[0,l]\rightarrow{\mathbb R}^n$ be a rectifiable curve and let
$B=\overline{B}\subset I,$ $l_{\gamma_1}(B)=0.$ Suppose that
$\gamma_2:I\rightarrow {\mathbb R}^n$ is a rectifiable curve on
$I\setminus B,$ and $\gamma_1(t)=\gamma_2(t)$ for $t\in B.$ Then
$\gamma_2$ is also rectifiable and $l_{\gamma_2}(B)=0.$ }
\end{propo}

\medskip
\begin{propo}\label{pr2*0}
{\em Let $\gamma:I=[0,l]\rightarrow{\mathbb R}^n$ be a  curve,
$B=\overline{B}\subset I,$ and let $E\subset I$ be a set with
$\overline{E}\subset E\cup B,$ and  $E\cap B=\varnothing.$ If
$\gamma$ is rectifiable on $I\setminus(E\cup B),$ and moreover, for
every point $t\in I\setminus B$ there exists a neighborhood $V,$ in
which $\gamma$ is rectifiable and $l_{\gamma}(V)=l_{\gamma}
(V\setminus E),$ then $\gamma$ is rectifiable on $I\setminus B$ and
$l_{\gamma} (I\setminus
B)=l_{\gamma}\left(I\setminus (E\cup B)\right).$ }
\end{propo}

\medskip
\begin{propo}\label{pr3*0}
{\em Let $\gamma:I\rightarrow{\mathbb R}^n$ be a rectifiable curve.
If $l_{\gamma}(B)=0$ for every $B\subset I$ with ${\rm
m_1\,}(B)=0,$ then $l_{\gamma}(t)$ is absolutely continuous
function.}
\end{propo}

\subsection{} Define for a mapping $f:D\rightarrow {\mathbb R}^n,$ a
set $E\subset D$ and a point $y\in {\mathbb R}^n,$ the
multiplicity function $N(y, f, E)$ as the number of preimages of
$y$ in $E,$ i.e.,
\begin{equation*}
N(y, f, E) = {\rm card}\,\left\{x \in E: f(x) =
y\right\}\,,
\end{equation*}
and put
\begin{equation*}
N(f,E)=\sup\limits_{y\in{\mathbb
R}^n}\,N(y,f,E)\,.
\end{equation*}
Recall that a domain $G\Subset D$ is said to be {\it
normal domain of $f,$} if $\partial f(G)=f\left(\partial G\right).$
If $G$ is a normal domain, then $\mu(y, f, G)$ is a constant for any
$y\in f(G).$ This constant will be denoted by $\mu(f, G).$ Let
$f:D\rightarrow {\mathbb R}^n$ be a discrete open mapping, then $\mu(f,
G)=N(f,G)$ for every normal domain $G\subset D,$  see, e.g.
\cite[Proposition~I.4.10]{Ric93}.

\medskip
Let $V\subset D$ be a normal domain, and $f(V)=V^{\,*}.$
Following \cite{Pol70}, we define $g_V:V^{\,*}\rightarrow {\mathbb R}^n$ as follows: if $y\in
V^{\,*},$ and $f^{\,-1}(y)\cap V=\{x_i\},$ then
\begin{equation}\label{eq5.9A}
g_V(y)=\frac{1}{q} \sum\limits_{i}i(x_i,f)x_i\,,
\end{equation}
where $q=\sum\limits i(x_i,f)=\mu(f,V).$ This mapping can be regarded as the inverse for the nonhomeomorphic mapping $f.$  The following statement is due to \cite{Guo}.

\medskip
\begin{propo}\label{pr5*}
{\em Let
$f:D\rightarrow {\mathbb R}^n$ be an open discrete mapping, $f\in
W_{loc}^{1,n-1}(D),$ obeying $K_I(x,f)\in L_{loc}^{1},$ and
$m(f(B_f))=0.$ Then $g_V(y)$ is continuous $V^{\,*}$ and $g_V(y)\in
ACL^n (V^{\,*}).$}
\end{propo}

\medskip
Similar to $p$-inner dilatation, we define for any $x\in D$ and fixed $p\ge 1$ the $p$-outer dilatation of $f$ by
\begin{equation*}
K_{O, p}(x,f)\quad =\quad \left\{
\begin{array}{rr}
\frac{\Vert f^\prime(x)\Vert^p}{|J(x,f)|}, & J(x,f)\ne 0,\\
1,  &  f^{\,\prime}(x)=0, \\
\infty, & {\rm otherwise}.
\end{array}
\right.\,
\end{equation*}
In the case $p=n$ the classical outer dilation coefficient coincides with $p$-outer dilatation, i.e. $K_O(x, f)=K_{O, n}(x,f).$

\section{Main results}

Poletski\u\i\ lemma's yields that $ACP^{-1}$-property can be regarded as the well-known Fuglede theorem in the ``inverse direction''. Fuglede's theorem states that for a continuous mapping $f:D\rightarrow {\mathbb R}^n$ of Sobolev class $W^{1,p}_{\rm loc}(D)$ and a family $\Gamma$ of all locally rectifiable curves in $D$ having a closed subpath on which $f$ is not absolutely continuous, $\Gamma$ is exceptional with respect to $p$-module, i.e. $\mathcal M_p(\Gamma)=0.$ The following theorem extends Poletski\u\i\ lemma's to the mappings satisfying the assumptions of Theorems \ref{th1A} or \ref{th1B}. This theorem involves the $ACP_p^{-1}$-property.

\medskip
\begin{theo}\label{lem4.5.1}
{\em Let $f:D\rightarrow {\mathbb R}^n$
be an open discrete mapping satisfying the assumptions either of Theorem \ref{th1A} or of Theorem \ref{th1B}. Then $f\in ACP_p^{-1}.$}
\end{theo}

\begin{proof}
In view of Proposition \ref{pr3}, it suffices to prove that for any $p$-a.e. closed curve
$\widetilde{\gamma}$ such that $f\circ\gamma=\widetilde{\gamma},$ the
$f$-representation $\gamma^{\,*}$ of $\gamma$ with respect to
$\widetilde{\gamma}$ is rectifiable and absolutely continuous. We also can restrict ourselves by considering a subfamily of
$\widetilde{\Gamma},$ which belongs to a compact subdomain
$D^{\,\prime}$ of the domain $D.$ The general case is obtained from this by exhausting $\left\{V_i\right\}_{i=1}^{\infty}$ of the domain
$f(D)$ by compact subdomains $V_i\Subset D$.

\medskip
Denote by $\widetilde{\gamma}^{\,0}$ the normal representation of a
closed rectifiable curve $\widetilde{\gamma},$ which lies in $D^{\,\prime},$ $\widetilde{\gamma}=f\circ\gamma,$ and
by $\gamma^{\,*}$ the $f$-representation of $\gamma$ with respect
to $\widetilde{\gamma}.$ Now $\widetilde{\gamma}^{\,0}:[0,
l(\gamma)]\rightarrow {\mathbb R}^n,$ $\gamma^{\,*}:[0,
l(\gamma)]\rightarrow {\mathbb R}^n.$ Set $I:=[0, l(\gamma)].$

We show that for a.e. closed curve $\widetilde{\gamma},$ the curve
$\gamma^{\,*}$ is rectifiable on $I\setminus\gamma^{\,*}(B_f),$
where $\gamma^{\,*}(B_f)=\left\{s:\gamma^{\,*}(s)\in B_f\right\}.$
Let the sets $D^{\,\prime}\setminus B_f$ be covered by a countable system of
neighborhoods  $\left\{A_l\right\}$ at which the corresponding mapping $f_l=f|_{A_l}$ is
homeomorphic. Put $h_l=f_l^{\,-1}.$ Since under assumptions of Theorem \ref{lem4.5.1} $f\in W_{\rm loc}^{1, p},$
$p>n-1,$ one concludes that $h_l\in W_{\rm loc}^{1, 1}$ (\cite[Theorem~1.1]{Zie69}).

The following lemma shows that the mapping $h_l$ has richer differential properties.

\medskip
\begin{lemma}\label{lem4.5.2}
{\em Let $f:D\rightarrow {\mathbb R}^n$
be an open discrete mapping satisfying the assumptions either of Theorem \ref{th1A} or of Theorem \ref{th1B}. Then $h_l\in W_{\rm loc}^{1, p}.$}
\end{lemma}

\begin{proof}[Proof of Lemma \ref{lem4.5.2}]
By the assumptions, $f$ has $N$-property, therefore, $J(y, h_l)\ne 0$ for a.e. $y\in f(A_l).$
Consequently, $\Vert h_l^{\,\prime}(y)\Vert^n=K_O(y, h_l)\cdot |J(y,
h_l)|$ for a.e. $y\in f(A_l).$

Consider first the case $p\le n.$  By
\cite[Theorem~3.2.5]{Fe69} and the features of $N^{\,-1}$-property for $f$, we
have
\begin{equation}\label{eq6.1}
\begin{split}
\int\limits_{f(A_l)}\Vert h_l^{\,\prime}(y)\Vert^n\, dm(y)
&=\int\limits_{f(A_l)}K_O(y, h_l)\,|J(y, h_l)|\, dm(y)
\\&=\int\limits_{f(A_l)}K_O(h_l^{\,-1}(h_l(y)), h_l)\, |J(y, h_l)|\,
dm(y)
\\&=\int\limits_{A_l}K_O(h_l^{\,-1}(x), h_l)\,dm(x)\,.
\end{split}
\end{equation}
Since $f\in W_{\rm loc}^{1, p},$ $p>n-1,$ and $f$ is open, it is
differentiable a.e.; see e.g. \cite[Lemma~3]{Vai65}. Further,
$h_l^{\,\prime}(h_l^{\,-1}(x))=g^{\prime}(f(x)),$ hence,
$h_l^{\prime}(f(x))=\left(f^{\,\prime}(x)\right)^{\,-1}$ at all
points of nondegenerate differentiability of $f$. Since $f$ possesses
$N^{\,-1}$-property, we have $J(x, f)\ne 0$ a.e., and moreover,
\begin{equation}\label{eq6.2}
\Vert
h_l^{\,\prime}\left(f(x)\right)\Vert=\frac{1}{l\left(f^{\,\prime}(x)\right)}\,,\qquad
J(f(x), h_l)=\frac{1}{J(x, f)}\,;
\end{equation}
see \cite[Section~4.I]{Resh89}. Applying (\ref{eq6.1}) and
(\ref{eq6.2}), one gets
\begin{equation*}
\int\limits_{f(A_l)}\Vert h_l^{\,\prime}(y)\Vert^n\,
dm(y)=\int\limits_{A_l}K_I(x, f)\,dm(x)\,.
\end{equation*}
Thus, $h_l\in W_{\rm loc}^{1, n},$ because $h_l\in
W_{\rm loc}^{1, 1}.$ In particular, we proved that $h_l\in W_{\rm loc}^{1, p},$ for $n-1<p\le n.$

\medskip
Now let $p>n.$ Since $f$ satisfies $N$-property, $J(y, h_l)\ne 0$ for
a.e. $y\in f(A_l).$ Consequently, $\Vert
h_l^{\,\prime}(y)\Vert^p=K_{O, p}(y, h_l)\cdot |J(y, h_l)|$ for a.e.
$y\in f(A_l).$ Arguing similarly to above, one obtains
\begin{equation*}
\begin{split}
\int\limits_{f(A_l)}\Vert h_l^{\,\prime}(y)\Vert^p\, dm(y)
&=\int\limits_{f(A_l)}K_{O, p}(y, h_l)\, |J(y, h_l)|\, dm(y)
\\&=\int\limits_{f(A_l)}K_{O, p}(h_l^{\,-1}(h_l(y)), h_l)\, |J(y,
h_l)|\, dm(y)
\\&=\int\limits_{A_l}K_{O, p}(h_l^{\,-1}(x),
h_l)\,dm(x)\,.
\end{split}
\end{equation*}
Applying
(\ref{eq6.2}), one concludes
\begin{equation*}
\int\limits_{f(A_l)}\Vert h_l^{\,\prime}(y)\Vert^p\,
dm(y)=\int\limits_{A_l}K_{I, p}(x, f)\,dm(x)\,,
\end{equation*}
which yields  $h_l\in
W_{loc}^{1, p}$ for $p>n,$ because $h_l\in W_{loc}^{1, 1}.$ This completes the proof of Lemma \ref{lem4.5.2}.
\end{proof}

\medskip
We proceed the proof of Theorem \ref{lem4.5.1} and note that by \cite[section~28.2]{Vai71}, $h_l\in ACP_p.$ Observe that,
if $\gamma^{\,*}(s)\in A_l\cap A_j,$ then
$h_l\left(\widetilde{\gamma}^{\,0}(s)\right)=h_j\left(\widetilde{\gamma}^{\,0}(s)\right).$
Since $\widetilde{\gamma}^{\,0}$ is parameterized by $s,$ one
defines a mapping
$g:\widetilde{\gamma}^{\,0}|_{I\setminus\gamma^{\,*}(B_f)}\rightarrow
{\mathbb R}^n$ by $g(\widetilde{\gamma}^{\,0}(s))=h_k\left(\widetilde{\gamma}^{\,0}(s)\right),$ when $\gamma^{\,*}(s)\in A_k.$
Set $\frac{\partial g_l}{\partial y_j}\,(s)=\frac{\partial
h_{kl}}{\partial y_j} \left(\widetilde{\gamma}(s)\right).$ It follows from above that $\gamma^{\,*}$ is locally absolutely
continuous on every open interval of $I\setminus\gamma^{\,*}(B_f)$
for $p$-a.e. $\widetilde{\gamma}=f\circ\gamma^{\,*}.$ Applying Theorem~1.3(6) from
\cite{Vai71}, one gets
\begin{equation*}
\begin{split}
l_{\gamma^{\,*}}\left(I\setminus
\gamma^{\,*}(B_f)\right)&=\int\limits_{I\setminus
\gamma^{\,*}(B_f)}\, |\gamma^{\,*\,\prime}(s)|\,\,dm_1(s)
\\&\le \int\limits_{I\setminus\gamma^{\,*}(B_f)}\,
\left(\sum\limits_{l,j}\left(\frac{\partial g_l}{\partial y_j}(s)
\right)^2\right)^{1/2}\,dm_1(s)
\end{split}
\end{equation*}
for $p$-a.e. $\widetilde{\gamma}=f\circ\gamma^{\,*}.$ Since $h_l\in
W_{\rm loc}^{1,p}$ and $m\left(D^{\,\prime}\right)<\infty,$
\begin{equation*}
\int\limits_{I\setminus\gamma^{\,*}(B_f)}\,
\left(\sum\limits_{l,j}\left(\frac{\partial g_l}{\partial y_j}(s)
\right)^2\right)^{1/2}\,dm_1(s)\,<\,\infty
\end{equation*}
for $p$-a.e. curves $\widetilde{\gamma}=f\circ\gamma^{\,*}$ (cf.
\cite[Theorem~3(e)]{Fu57}). Consequently, $\gamma^{\,*}$ is
rectifiable on $I\setminus\gamma^{\,*}(B_f)$ for $p$-a.e. curves
$\widetilde{\gamma}\in \widetilde{\Gamma}.$

Moreover, $l_{\gamma^{\,*}}(C)=0$ for $p$-a.e. curves
$\widetilde{\gamma}$ and every set $C,$ $C\subset
I\setminus\gamma^{\,*}(B_f),$ such that $m_1(C)=0.$ In fact,
$l_{\gamma^{\,*}}(C)=\int\limits_{C}
|\gamma^{\,*\,\prime}(s)|\,\,dm_1(s)=0$
for $p$-a.e. curves $\widetilde{\gamma}=f\circ\gamma^{\,*}.$

\medskip
Now let $B_l$ be a subset of the branched set $B_f$ of $f$ with $i(x,f)=l,$
and
\begin{equation*}
\gamma^{\,*}(B_l)=\left\{s\in I: \gamma^{\,*}(s)\in
B_l\right\}.
\end{equation*}

We show that for $p$-a.e. $\widetilde{\gamma}$
the curve $\gamma^{\,*}$ is rectifiable on
$I\setminus\bigcup\limits_{k\,>\,l}\gamma^{\,*}(B_k)$ for all $l\in
{\mathbb N}$ and $l_{\gamma^{\,*}}(C)=0$ for every set $C$ with
$m_1(C)=0$ and $C\subset
I\setminus\bigcup\limits_{k\,>\,l}\gamma^{\,*}(B_k).$ This is obtained by induction with respect to $l.$ The case $l=1$ has been proved above.

Now
we assume that this property holds for $l=j$ and show its validness for $l=j+1.$

Since $f$ has $N^{-1}$-property, $J(x, f)\ne 0$
a.e.. Hence, since $f$ is differentiable a.e.,
we have $m(B_f)=0$ (see \cite[Lemma~2.14]{MRV69}).

Cover the set $B_j$ by at most countable system of normal domains
$\left\{U_l\right\}_{l=1}^{\infty}$ such that $\mu\left(f,
U_l\right)=j$ with $\mu\left(f, U_l\right)=N(f, U_l)$ (see \cite[Lemma~2.9]{MRV69}). By Proposition
\ref{pr5*}, the mapping $g_{l}=g_{U_l}$ defined by (\ref{eq5.9A}) is
absolutely continuous for $p$-a.e. curves in $U^*_l=f(U_l)$. Observe also that
$g_l\left(\widetilde{\gamma}^{\,0}(s)\right)=\gamma^{\,*}(s),$
whenever $\widetilde{\gamma}^{\,0}(s)\in f\left(B_j\cap U_l\right).$

Letting $\alpha_{k,l}:=g_l\left(\widetilde{\gamma}^{(k)}(s)\right),$
where
$\widetilde{\gamma}^{\,0}|_{f(U_l)}=\bigcup\limits_{k=1}^{\infty}\widetilde{\gamma}^{(k)}(s),$
$l=1,2,\ldots\,,$ we have by absolute continuity of $g_l$ on $p$-a.e.
curves that $l_{\alpha_{k,l}}\left(\alpha_{k,l} \left(B_j\cap
U_l\right)\right)=0$ for $p$-a.e. $\widetilde{\gamma}=f\circ\gamma.$ Now, by Proposition \ref{pr1*0},
\begin{equation*}
l_{\gamma^{\,*}}\left(\gamma^{\,*}\left(B_j\cap
U_l\right)\right)=0
\end{equation*}
for $p$-a.e. curves $\widetilde{\gamma}$ and $\gamma.$ Summation over
all neighborhoods $U_l$ yields
$l_{\gamma^{\,*}}(\gamma^{\,*}(B_j))=0$ for $p$-a.e. curves
$\widetilde{\gamma}=f\circ\gamma^{\,*}.$

Take in Proposition \ref{pr2*0}
$B:=\bigcup\limits_{k\,>\,j}\gamma^{\,*}(B_k)$ and
$E:=\gamma^{\,*}(B_j).$ The local topological index $i(x, f)$ is
lower semicontinuous (see, e.g.
\cite[Lemma~VI.8.13]{Ric93}). Now $B$ is closed, and by the induction
assumption $\gamma^{\,*}$ is rectifiable on $I\setminus
\bigcup\limits_{k\,>\,j}\gamma^{\,*}(B_k),$ and
$l_{\gamma^{\,*}}(C)=0$ for $C\subset I\setminus
\bigcup\limits_{k\,>\,j}\gamma^{\,*}(B_k)$ and $m_1 (C)=0.$
Since $D^{\,\prime}$ is a compact, there exists $M\in {\mathbb N}$ with
$i(x,f)\le M.$ By the previous step,
$l_{\gamma^{\,*}}\left(\bigcup\limits_{j\,=\,2}^{M}\gamma^{\,*}(B_j)\right)=0$
for $p$-a.e.
$\widetilde{\gamma}=f\circ\gamma^{\,*}.$ Hence, by Proposition \ref{pr3*0},
$\gamma^{\,*}$ is absolutely continuous and rectifiable for $p$-a.e.
closed curves $\widetilde{\gamma}=f\circ\gamma^{\,*},$
$\widetilde{\gamma}\in\widetilde{\Gamma}.$ This completes the proof of Theorem \ref{lem4.5.1}.
\end{proof}

\medskip
Now both Theorems \ref{th1A} and \ref{th1B} follow from Theorem
\ref{lem4.5.1} and Theorem~3.1 from \cite{SS14}.

\medskip
\medskip
{\small \leftline{\textbf{Anatoly Golberg}} \em{
\leftline{Department of Applied Mathematics,} \leftline{Holon
Institute of Technology,} \leftline{52 Golomb St., P.O.B. 305,}
\leftline{Holon 5810201, ISRAEL} \leftline{Fax: +972-3-5026615}
\leftline{e-mail: golberga@hit.ac.il}}}

\medskip
{\small \leftline{\textbf{ Evgeny Sevost'yanov}} \em{
\leftline{Department of Mathematical Analysis,} \leftline{Zhitomir
State University,} \leftline{40 Bol'shaya Berdichevskaya Str.}
\leftline{Zhitomir 10 008, UKRAINE} \leftline{Phone: +38 -- (066) --
959 50 34} \leftline{e-mail: esevostyanov2009@mail.ru}}}

\end{document}